\documentclass{amsart}
\usepackage[utf8]{inputenc}
\usepackage{amsmath}
\usepackage{amssymb}
\usepackage{caption}
\usepackage{amsthm}
\usepackage[hidelinks]{hyperref}
\usepackage{cleveref}
\usepackage[justification=centering]{caption}
\crefname{lemma}{Lemma}{Lemmas}
\crefname{theorem}{Theorem}{Theorems}
\usepackage{xcolor}
\usepackage{comment}
\usepackage{graphicx}
\usepackage{tikz}
\usepackage{tikz-cd}
\usepackage{mathrsfs}
\usepackage{enumitem}
\usetikzlibrary{shapes.geometric}
\makeatletter
\def\@settitle{\begin{center}%
  \baselineskip14\p@\relax
  \bfseries
  \uppercasenonmath\@title
  \@title
  \ifx\@subtitle\@empty\else
     \\[1ex]\uppercasenonmath\@subtitle
     \footnotesize\mdseries\@subtitle
  \fi
  \end{center}%
}
\def\subtitle#1{\gdef\@subtitle{#1}}
\def\@subtitle{}
\makeatother

\newtheorem{theorem}{Theorem}[section]

\newtheorem{lemma}[theorem]{Lemma}

\theoremstyle{remark}
\newtheorem{remark}[theorem]{Remark}

\theoremstyle{remark}

\usepackage{chngcntr}
\usepackage{graphicx} 
\usepackage{float}
\counterwithout{equation}{section}
\counterwithout{theorem}{section}

\begin{document}

\title{Systoles of hyperbolic hybrids}

\author{Sami Douba}

\begin{abstract}
We exhibit closed hyperbolic manifolds with arbitrarily small systole in each dimension that are not quasi-arithmetic in the sense of Vinberg, and are thus not commensurable to those constructed by Agol, Belolipetsky--Thomson, and Bergeron--Haglund--Wise. This is done by taking hybrids of the manifolds constructed by the latter authors. 
\end{abstract}

\address{Institut des Hautes \'Etudes Scientifiques, 
Universit\'e Paris-Saclay, 35 route de Chartres, 91440 Bures-sur-Yvette, France}
\email{douba@ihes.fr}

\maketitle

Given a lattice $\Gamma$ in $\mathrm{Isom}(\mathbb{H}^n)$, the {\em systole} of the hyperbolic orbifold $M = \Gamma \backslash \mathbb{H}^n$ is the minimal length of a closed geodesic in $M$, i.e., the minimal translation length of a loxodromic element in $\Gamma$. Using the geometry of hyperbolic right-angled polygons, one can easily construct for each sufficiently small $\epsilon > 0$ a closed oriented hyperbolic surface of genus $2$ and systole precisely $\epsilon$. By Mostow rigidity, the topology of a closed hyperbolic $n$-manifold determines the systole as soon as $n\geq 3$; in particular, the systoles of such manifolds take on only countably many values. It is nevertheless true that for every $n \geq 3$ and every $\epsilon > 0$, there is a closed hyperbolic $n$-manifold of systole $< \epsilon$. For $n = 3$, this follows from Thurston's theory of hyperbolic Dehn filling.\footnote{In retrospect, there is a more elementary construction again involving hyperbolic right-angled polyhedra \cite{bogachevdouba}, though this also can ultimately be conceptualized as resulting from ``Dehn filling'' of Coxeter $3$-polyhedra.} 
No such tool is available in dimension~$4$, and it was in this dimension that Agol \cite{agol2006systoles} provided a strategy for producing closed hyperbolic manifolds of arbitrarily small systole by ``inbreeding'' arithmetic manifolds. The separability arguments required to implement Agol's strategy in every dimension were later supplied independently by Belolipetsky--Thomson \cite{MR2821431} and Bergeron--Haglund--Wise~\cite{bergeron2011hyperplane}. 

An interesting feature of the manifolds arising from this ``inbreeding'' construction is that they are all quasi-arithmetic in the sense of Vinberg \cite{MR207853}, as observed by Thomson \cite{MR3451458}. If $\Gamma < \mathrm{Isom}(\mathbb{H}^n)$ is a lattice, then $\Gamma$ is said to be {\em quasi-arithmetic} if there is a totally real number field $k \subset \mathbb{R}$ and a $k$-group ${\bf G}$ such that ${\bf G}(\mathbb{R})$ is isogenous to $\mathrm{Isom}(\mathbb{H}^n)$ with $\Gamma$ virtually contained in ${\bf G}(k)$ via this isogeny, but ${\bf G}^\sigma(\mathbb{R})$ is compact for every embedding $\sigma: k \rightarrow \mathbb{R}$ distinct from the inclusion $k \subset \mathbb{R}$. In this case, the adjoint trace field of $\Gamma$ coincides with $k$, and $\Gamma$ is arithmetic if and only if $\mathrm{trAd}(\gamma)$ is an algebraic integer for each $\gamma \in \Gamma$. 

The purpose of this note is to demonstrate the following. 

\begin{theorem}\label{main}
For each $n \geq 2$ and $\epsilon > 0$, there is a closed hyperbolic $n$-manifold of systole $< \epsilon$ that is not quasi-arithmetic.
\end{theorem}

It follows that the manifolds $M$ with small systole that we exhibit below are not commensurable to those constructed by Agol, Belolipetsky--Thomson, and Bergeron--Haglund--Wise. These $M$ are nevertheless pseudo-arithmetic in the sense of Emery and Mila \cite{MR4237964}, and hence we do not address the question of the latter authors as to whether every lattice in $\mathrm{Isom}(\mathbb{H}^n)$ for $n \geq 4$ is pseudo-arithmetic.

\begin{comment}
\begin{lemma}
Let $\Gamma_1, \Gamma_2$ be quasi-arithmetic lattices in a semisimple Lie group $G$, and suppose the $\Gamma_i$ share the same adjoint trace field $k$. If $\Gamma_1 \cap \Gamma_2$ is Zariski-dense in $G$, then the ambient groups of the $\Gamma_i$ are $k$-isogenous.
\end{lemma}
\end{comment}

\begin{proof}[Proof~of~Theorem~\ref{main}]
Set $k = \mathbb{Q}(\sqrt{2})$, and let $a \in \mathbb{Q}_{>0}$ be a non-square in $k$. For $n \geq 2$, let $f_1$ and~$f_2$ be the quadratic forms in $n+1$ variables with coefficients in $k$ given by
\begin{alignat*}{4}
x_1^2 + x_2^2 + \ldots + x_n^2 - \sqrt{2}x_{n+1}^2, \\
ax_1^2+x_2^2 +\ldots +x_n^2 - \sqrt{2}x_{n+1}^2,
\end{alignat*}
respectively. For $i=1,2$, denote by $\mathbb{H}_{f_i}$ the $f_i$-hyperboloid model of $\mathbb{H}^n$, and by $\mathrm{O}'(f_i; \mathbb{R}) = \mathrm{Isom}(\mathbb{H}_{f_i})$ the index-$2$ subgroup of $\mathrm{O}(f_i; \mathbb{R})$ preserving $\mathbb{H}_{f_i}$. Let $A_i$ be the $1$-dimensional subgroup of $\mathrm{O}'(f_i; \mathbb{R})$ consisting of all matrices of the form
\begin{alignat*}{4}
\begin{pmatrix} 
* &   & * \\
& I_{n-1} & \\
* &   & *
\end{pmatrix}.
\end{alignat*}
Since the $k$-points of $A_i$ are dense in the identity component $A_i^\circ$ of $A_i$, we can find $g_i \in A_i^\circ$ with entries in $k$ such that the leading eigenvalue $\lambda_i$ of $g_i$ satisfies $\mathrm{log}(\lambda_i) < \epsilon/4$. Denote by $H_i \subset \mathbb{H}_{f_i}$ the geodesic hyperplane $\{x_1 = 0\}$. Belolipetsky and Thomson \cite{MR2821431} show that there is an ideal $I \subset \mathbb{Z}[\sqrt{2}]$ such that for $i=1,2$, the hyperplanes $H_i$ and $g_i H_i$ project to disjoint embedded $2$-sided totally geodesic hypersurfaces $\Sigma_i$ and $\Sigma_i'$, respectively, in the orbifold $M_i := \Gamma_i(I) \backslash \mathbb{H}_{f_i}$, where~$\Gamma_i(I)$ denotes the principal congruence subgroup of level $I$ in $\Gamma_i := \mathrm{O}'(f_i ; \mathbb{Z}[\sqrt{2}])$. Up to diminishing $I$, we can assume moreover that the~$\Gamma_i(I)$ are torsion-free, so that the~$M_i$ are manifolds, and that the projection~$\omega_i$ to $M_i$ of the orthogeodesic $\tilde{\omega}_i \subset \mathbb{H}_{f_i}$ joining the hyperplanes~$H_i$ and~$g_i H_i$ does not cross~$\Sigma_i$ or~$\Sigma_i'$. Let $N_i$ be the component containing~$\omega_i$ of the hyperbolic manifold with totally geodesic boundary obtained by cutting $M_i$ along~$\Sigma_i$ and $\Sigma_i'$. Identify~$\Sigma_i$ and $\Sigma_i'$ with the two boundary components of $N_i$ joined by $\omega_i$. Finally, let $M$ be the manifold obtained by first gluing the $N_i$ along the~$\Sigma_i$ via a canonical isometry in the language of Mila \cite{MR4230399}, so that the $\omega_i$ glue to an orthogeodesic $\omega$ joining the~$\Sigma_i'$ of length $\mathrm{log}(\lambda_1)+ \mathrm{log}(\lambda_2)$, and then doubling the resulting manifold along its boundary; see Figure \ref{fig:systolefigure}. Then~$M$ contains a closed geodesic of length $2(\mathrm{log}(\lambda_1)+ \mathrm{log}(\lambda_2)) < \epsilon$, namely, the double of $\omega$. To see that~$M$ is not quasi-arithmetic, one can for instance use the fact that a Zariski-dense subgroup of a quasi-arithmetic lattice has the same adjoint trace field as the lattice itself; see \cite[Cor.~2]{bdr}. In particular, since the $\pi_1(N_i)$ are Zariski-dense (see \cite[Lem.~1.7.A]{MR932135}), they have adjoint trace field $k$. On the other hand, as explained in \cite{MR4230399}, the adjoint trace field of $\pi_1(M)$ is $ k(\sqrt{a}) \neq k$. Since $\pi_1(N_1) \subset \pi_1(M)$, we conclude that $\pi_1(M)$ is not quasi-arithmetic.
\end{proof}

\begin{figure}[htbp]
    \centering
        \includegraphics[clip, trim=0.5cm 21cm 0.5cm 1cm, width=1.00\textwidth]{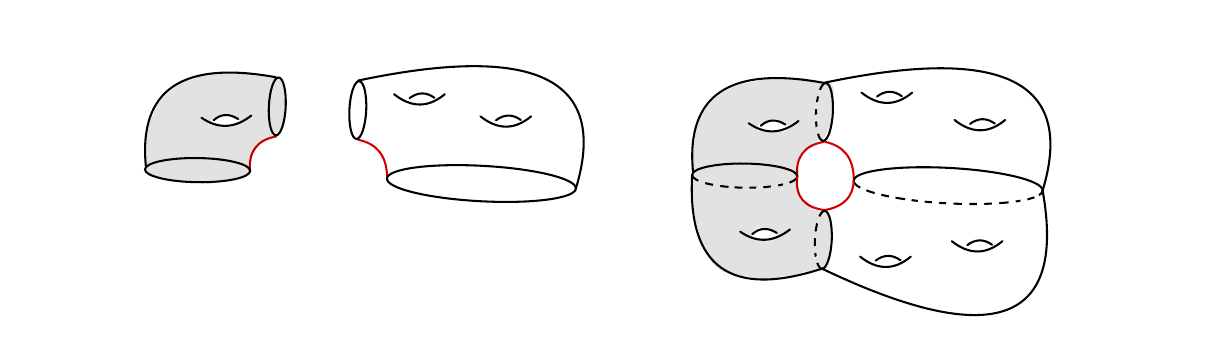}
    \caption{A schematic of the construction of the manifold $M$ in the proof of Theorem \ref{main}.}
    \label{fig:systolefigure}
\end{figure}

\begin{remark}\label{classicalhybrid}

We define a {\em classical hybrid} to be any closed hyperbolic orbifold $V$ obtained by gluing two (possibly disconnected) compact hyperbolic orbifolds $O_1$ and $O_2$ with totally geodesic boundary along their boundaries via isometries, such that each component of the double $2O_i$ of $O_i$ along $\partial O_i$ is arithmetic for $i=1,2$, and such that $2O_1$ is not commensurable to $2O_2$. Note that to say $2O_i$ is arithmetic is to say that there is an arithmetic lattice containing $\pi_1(O_i)$ and the reflections in the boundary components of the universal cover $\widetilde{O_i}$, so that the particular examples of nonarithmetic hyperbolic manifolds constructed by Gromov and Piatetski-Shapiro in \cite[2.8.C~and~2.9]{MR932135} are classical hybrids in the above sense.

For fixed $D$ and $\mu$, there are only finitely many monic integer polynomials of degree $\leq D$ and Mahler measure $\leq \mu$. Thus, for fixed $d$ and $n$, there is some $\epsilon_{d,n}$ such that any arithmetic hyperbolic $n$-orbifold with adjoint trace field of degree $\leq d$ has systole $\geq \epsilon_{d,n}$. It then follows that any classical hybrid of dimension~$n$ with adjoint trace field of degree $\leq d$ has systole $\geq \epsilon_{d,n}$. Indeed, let $\gamma$ be a closed geodesic in a classical hybrid $V$ as above. If $\gamma$ is contained entirely in $O_i$ for some $i=1,2$, then $\gamma$ has length $\geq \mathrm{sys}(2O_i) \geq \epsilon_{d,n}$. Otherwise,
\begin{alignat*}{4}
\mathrm{length}(\gamma) = \mathrm{length}(\gamma \cap O_1) + \mathrm{length}(\gamma \cap O_2) 
\geq \frac{1}{2} \mathrm{sys}(2O_1) +  \frac{1}{2} \mathrm{sys}(2O_2) 
%&\geq \frac{1}{2} \epsilon_{d,n} + \frac{1}{2} \epsilon_{d,n} \\
\geq \epsilon_{d,n}.
\end{alignat*}

By \cite[Lem.~3.3]{MR3090707}, for $n \geq 3$, any $n$-orbifold commensurable to a classical hybrid is again a classical hybrid. Since all the manifolds~$M$ as in the proof of Theorem~\ref{main} share the same adjoint trace field $k(\sqrt{a})$, we obtain from the above discussion that for $n \geq 3$ and $\epsilon < \epsilon_{4,n}$, the output $n$-manifold $M$ is not commensurable to any classical hybrid in the above sense. 
\end{remark}

\begin{remark}
The approach in the proof of Theorem \ref{main} can also be used to construct a lattice in $\mathrm{O}'(n,1)$ for each $n \geq 2$ that is neither quasi-arithmetic nor has integral traces.\footnote{Note that, even having fixed the adjoint trace field, small systoles are not enough to guarantee nonintegral traces, since the ambient groups of the relevant lattices are not admissible.} For instance, if in the proof of Theorem \ref{main} we pick the $g_i$ so that~$\lambda_1\lambda_2$ is not an algebraic integer (regardless of how small the $\log(\lambda_i)$ are), then $\Gamma = \pi_1(M)$ has the desired property. Indeed, by an argument similar to \cite[Lem.~1.1]{MR3892248}, up to conjugation within $\mathrm{O}'(f_1 ; \mathbb{R})$, we have $\Gamma \subset \mathrm{O}'(f_1; K)$, where $K = k(\sqrt{a})$. Thus, if $\mathrm{tr}\mathrm{Ad} (g) \in \mathcal{O}_K$ for each $g \in \Gamma$, then the same is true of $\mathrm{tr}(g)$ for each $g \in \Gamma$ by \cite[Theorems~1~and~2~and~Lem.~2]{MR279206}. On the other hand, if $\gamma \in \Gamma$ is an element representing the double of $\omega$ in $M$, then since $\gamma$ is a pure hyperbolic element, we have that $\mathrm{tr}(\gamma)$ is a monic palindromic Laurent polynomial in $\lambda_1\lambda_2$ with integer coefficients, so that if $\mathrm{tr}(\gamma)$ is an algebraic integer, then the same is true of $\lambda_1\lambda_2$.

For instance, one can take $a = 3$ and
\begin{alignat*}{4}
g_1 &= \begin{pmatrix} 3+2\sqrt{2} & & 4 + 2\sqrt{2} \\ & I_{n-1} & \\ 2+2\sqrt{2} & & 3+2\sqrt{2} \end{pmatrix}, \\
g_2 &= \begin{pmatrix} \frac{11+6\sqrt{2}}{7} & & \frac{4+6\sqrt{2}}{7} \\ & I_{n-1} & \\ \frac{18+6\sqrt{2}}{7} & & \frac{11+6\sqrt{2}}{7}  \end{pmatrix}.
\end{alignat*}
\end{remark}

One can arrange in the proof of Theorem \ref{main} for the double of $\omega$ to be the shortest closed geodesic in $M$. Indeed, a similar approach yields the following (compare, for instance, \cite[Thm.~1.5]{MR3702548}). 

\begin{theorem}\label{shortest}
For every $n \geq 3$ and $m \in \mathbb{Z}_{>0}$, one can find $\epsilon_m < \frac{1}{m}$ and $m$ pairwise incommensurable closed hyperbolic $n$-manifolds of the same volume each possessing a unique shortest closed geodesic of length precisely $\epsilon_m$. 
\end{theorem}

\begin{proof}
Fix $n$ and $m$. In the proof of Theorem \ref{main}, take $\epsilon \leq 2^{-m}\min\{\frac{1}{m}, \epsilon_{2,n}\}$, where~$\epsilon_{2,n}$ is as in Remark \ref{classicalhybrid}. For $i=1,2$, let $p_i \in H_i$ be the point at which the axis of $g_i$ meets $H_i$, and let $q_i$ be the midpoint of the geodesic segment $\tilde{\omega}_i \subset \mathbb{H}_{f_i}$. For a point $y_i \in \mathbb{H}_{f_i}$ and a subgroup $\Delta_i < \Gamma_i$, let $D_i(\Delta_i, y_i)$ be the Dirichlet domain for~$\Delta_i$ in $\mathbb{H}_{f_i}$ centered at $y_i$. Given an ideal $J \subset \mathbb{Z}[\sqrt{2}]$, denote by $\Delta_i(J)$ (resp., $\Delta_i'(J)$) the stabilizer in $\Gamma_i(J)$ of $H_i$ (resp., $gH_i$), and let $\Phi_i(J) = \langle \Delta_i(J), \Delta_i'(J) \rangle < \Gamma_i(J)$. We pass to an ideal $J \subset I$ such that
\begin{enumerate}[label=(\roman*)]
\item\label{gf} the walls of $D_i(\Delta_i(J), p_i)$  and the walls of $D_i(\Delta_i'(J), gp_i)$ together comprise the walls of $D_i(\Phi_i(J), q_i)$;

\item\label{bananas} no wall of $D_i(\Delta_i(J), p_i)$ enters the $\epsilon_{2,n}$-neighborhood of $gH_i$ in $\mathbb{H}_{f_i}$, and no wall of $D_i(\Delta_i'(J), gp_i)$ enters the $\epsilon_{2,n}$-neighborhood of $H_i$.
\end{enumerate}
From \ref{gf}, we have that $\Phi_i(J)$ is geometrically finite, and hence separable in $\Gamma_i(J)$ by aforementioned work of Bergeron--Haglund--Wise \cite{bergeron2011hyperplane}. It follows that there is a finite-index subgroup $\Lambda_i < \Gamma_i(J)$ containing $\Phi_i(J)$ such that no wall of $D_i(\Lambda_i, q_i)$ that is not a wall of $D_i(\Phi_i(J), q_i)$ enters the $\epsilon_{2,n}$-neighborhood of $H_i \cup gH_i$. We replace~$M_i$ with the manifold $\Lambda_i \backslash \mathbb{H}_{f_i}$.
Then any orthogeodesic segment in~$M_i$ with endpoints on $\Sigma_i \cup \Sigma_i'$ apart from $\omega_i$ now has length $\geq \epsilon_{2,n}$.

Let $N_i'$ be the double of $N_i$ across all boundary components of $N_i$ apart from~$\Sigma_i$, and $\omega_i' \subset N_i'$ the double of $\omega_i$. Then any connected closed manifold $C$ built out of a total of~$2^m$ copies of the~$N_i'$ by gluing boundary components via isometries so that the copies of $\omega_i'$ match up has a single shortest closed geodesic, namely, the closed geodesic obtained by gluing up all the copies of $\omega_i'$, and this geodesic has length $< 2^m\frac{\epsilon}{2} < \frac{1}{m}$. Moreover, the systole and volume of such $C$ depend only on the number of times each of $N_1'$ and~$N_2'$ features in the construction of $C$.

Now we claim that if we set $a = 17$ in the proof of Theorem \ref{main}, then the manifolds~$C$ as above represent at least $m$ distinct commensurability classes, even under the stipulation that $N_1'$ and $N_2'$ feature equally in the construction of $C$. Indeed, for our choice of $a$, Gelander and Levit \cite{MR3261631} show that the forms $f_1$ and $f_2$ are not similar over~$k$. Since $n \geq 3$, it then follows from work of Raimbault\footnote{Raimbault's work \cite{MR3090707} deals with manifolds constructed cyclically out of pieces of {\em arithmetic} manifolds the ambient groups of whose associated lattices are distinct. However, his arguments persist if ``arithmetic'' is weakened to ``quasi-arithmetic'' in the previous sentence, as observed in~\cite{bdr}, and as is the case in our context.} \cite{MR3090707} that if~$C_1$ and $C_2$ are commensurable manifolds of the form~$C$ above, and $\alpha_j$ is the cyclic sequence of $1$'s and $2$'s dictating the decomposition of $C_j$ into copies of $N_1'$ and $N_2'$ for $j=1,2$, then $\alpha_1$ and $\alpha_2$ are related by a dihedral symmetry. Now it suffices to observe that one can easily produce $m$ cyclic sequences of $1$'s and $2$'s of length~$2^m$ no two of which are related by such a symmetry and each of which features an equal number of $1$'s and $2$'s.
\end{proof}

\begin{remark}
In the proof of Theorem \ref{main}, if there is no other orthogeodesic in $M_i = \Gamma_i(I) \backslash \mathbb{H}_{f_i}$ joining $\Sigma_i$ to $\Sigma_i'$ of precisely the same length as $\omega_i$, then the argument of Belolipetsky and Thomson \cite{MR2821431} (more precisely, the proof of their generalized Margulis--Vinberg lemma \cite[Lem.~3.1]{MR2821431}) in fact shows that, for any $\delta > 0$, it is possible to pass to an ideal $I_\delta \subset I \subset \mathbb{Z}[\sqrt{2}]$ such that, replacing $M_i$ with $\Gamma_i(I_\delta) \backslash \mathbb{H}_{f_i}$, every orthogeodesic in $M_i$ with endpoints on $\Sigma_i \cup \Sigma_i'$ apart from $\omega_i$ now has length $\geq \delta$. Taking $\delta = \epsilon_{2,n}$, one could then instead use the above congruence covers $M_i$ in the proof of Theorem \ref{shortest}. It has been suggested to us by Misha Belolipetsky that this can potentially be arranged for an appropriate (indeed, generic) choice of $g_i$, but we have not pursued this here.
\end{remark}

\subsection*{Acknowledgements} I am grateful to Misha Belolipetsky, Nikolay Bogachev, and Jean Raimbault for instructive discussions, and for their comments on an earlier draft of this note. The author was supported by the Huawei Young Talents Program.

\bibliography{bantingbib}{}

\begin{thebibliography}{10}

\bibitem{agol2006systoles}
{\sc I.~Agol}, {\em Systoles of hyperbolic 4-manifolds}, arXiv preprint
  math/0612290,  (2006).

\bibitem{MR2821431}
{\sc M.~V. Belolipetsky and S.~A. Thomson}, {\em Systoles of hyperbolic
  manifolds}, Algebr. Geom. Topol., 11 (2011), pp.~1455--1469.

\bibitem{bergeron2011hyperplane}
{\sc N.~Bergeron, F.~Haglund, and D.~T. Wise}, {\em Hyperplane sections in
  arithmetic hyperbolic manifolds}, J. Lond. Math. Soc. (2), 83 (2011),
  pp.~431--448.

\bibitem{bogachevdouba}
{\sc N.~Bogachev and S.~Douba}, {\em {Geometric and arithmetic properties of
  L\"obell polyhedra}}, arXiv preprint arXiv:2304.12590,  (2023).

\bibitem{bdr}
{\sc N.~Bogachev, S.~Douba, and J.~Raimbault}, {\em {Infinitely many
  commensurability classes of compact Coxeter polyhedra in $\mathbb{H}^4$ and
  $\mathbb{H}^5$}}, arXiv preprint arXiv:2309.07691,  (2023).

\bibitem{MR4237964}
{\sc V.~Emery and O.~Mila}, {\em Hyperbolic manifolds and
  pseudo-arithmeticity}, Trans. Amer. Math. Soc. Ser. B, 8 (2021),
  pp.~277--295.

\bibitem{MR3261631}
{\sc T.~Gelander and A.~Levit}, {\em Counting commensurability classes of
  hyperbolic manifolds}, Geom. Funct. Anal., 24 (2014), pp.~1431--1447.

\bibitem{MR932135}
{\sc M.~Gromov and I.~Piatetski-Shapiro}, {\em Nonarithmetic groups in
  {L}obachevsky spaces}, Inst. Hautes \'{E}tudes Sci. Publ. Math.,  (1988),
  pp.~93--103.

\bibitem{MR3892248}
{\sc O.~Mila}, {\em Nonarithmetic hyperbolic manifolds and trace rings},
  Algebr. Geom. Topol., 18 (2018), pp.~4359--4373.

\bibitem{MR4230399}
\leavevmode\vrule height 2pt depth -1.6pt width 23pt, {\em The trace field of
  hyperbolic gluings}, Int. Math. Res. Not. IMRN,  (2021), pp.~4392--4412.

\bibitem{MR3702548}
{\sc C.~Millichap}, {\em Mutations and short geodesics in hyperbolic
  3-manifolds}, Comm. Anal. Geom., 25 (2017), pp.~625--683.

\bibitem{MR3090707}
{\sc J.~Raimbault}, {\em A note on maximal lattice growth in {${\rm
  SO}(1,n)$}}, Int. Math. Res. Not. IMRN,  (2013), pp.~3722--3731.

\bibitem{MR3451458}
{\sc S.~Thomson}, {\em Quasi-arithmeticity of lattices in
  {$\mathrm{PO}(n,1)$}}, Geom. Dedicata, 180 (2016), pp.~85--94.

\bibitem{MR207853}
{\sc E.~B. Vinberg}, {\em Discrete groups generated by reflections in
  {L}oba\v{c}evski\u{\i} spaces}, Mat. Sb. (N.S.), 72(114) (1967),
  pp.~471--488; correction, ibid. 73 (115) (1967), 303.

\bibitem{MR279206}
\leavevmode\vrule height 2pt depth -1.6pt width 23pt, {\em Rings of definition
  of dense subgroups of semisimple linear groups}, Izv. Akad. Nauk SSSR Ser.
  Mat., 35 (1971), pp.~45--55.

\end{thebibliography}
\bibliographystyle{siam}

\end{document}